\documentclass[11pt,a4paper,oneside,DIV=12,headings=small]{scrartcl}
\usepackage[english]{babel}
\usepackage{graphicx}
\usepackage{microtype}
\usepackage{ellipsis}

\usepackage{pgfplots,tikz,subfigure}

\usepackage{amssymb,amsmath,amsfonts,amsthm,changes}

\newtheorem{theorem}{Theorem}[section]

\newtheorem{proposition}[theorem]{Proposition}

\theoremstyle{definition}
\newtheorem{definition}[theorem]{Definition}
\theoremstyle{remark}
\newtheorem{remark}[theorem]{Remark}
\newtheorem{example}[theorem]{Example}

\DeclareMathOperator{\Tr}{Tr}
\DeclareMathOperator{\T}{T}
\DeclareMathOperator{\diag}{diag}
\newcommand{\deri}{\,\mathrm{d}}

\usepackage[]{graphicx}

\title{Mixed control of vibrational systems}

\author{Ivica Naki\'{c}\thanks{University of Zagreb, Department of Mathematics, Croatia} 
\and Zoran Tomljanovi\'{c}\thanks{Department of Mathematics, J. J. Strossmayer University of Osijek, Croatia} 
\and Ninoslav Truhar\footnotemark[2]
}

\date{\vspace{-5ex}}

\begin{document}

\maketitle                   

\begin{abstract}
We consider new performance measures for vibrational systems  based on the $H_2$ norm of linear time invariant systems. New measures will be used as an optimization criterion for the optimal damping of vibrational systems. We consider both theoretical and concrete cases in order to show how new measures stack up against the standard measures. The quality and advantages of new measures as well as the behaviour of optimal damping positions and corresponding damping viscosities are illustrated in numerical experiments.
\end{abstract}






\section{Introduction} 
\label{sec:introduction}
In this paper we are concerned with the minimization of vibrations of an abstract time invariant vibrational system \cite{karl1994control} described by
\begin{equation}
	\label{eq:main}
	G = \left\{
	\begin{aligned}
	& M \ddot{q} + D \dot{q} + K q = B_2 u, \\
	& q(0) = q_0, \; \dot{q}(0) = \dot{q}_0,\\
	& y = \begin{bmatrix}
	C_1 q \\ C_2 \dot{q}
	\end{bmatrix}.
	\end{aligned}
	\right.
\end{equation}
Here $M$ denotes the mass matrix, $D$ denotes the damping matrix and $K$ denotes the stiffness matrix. We assume that $M$ is a non-singular matrix and that matrices $M$, $D$ and $K$ are real matrices of order $n$.
Vector $q\in \mathbb{R}^n$ contains the state variables, while the vector $y\in \mathbb{R}^{r}$ denotes the observed output which together with the output matrices $C_1, C_2 \in \mathbb{R}^{p\times n}$ determine the system displacements and velocities of interest. System disturbances are denoted by the vector $u\in \mathbb{R}^{m}$ and the matrix  $B_2\in \mathbb{R}^{n\times m }$. The first equation in \eqref{eq:main} is usually called the state equation.

Systems of the form \eqref{eq:main} are used as a linearized model for a large class of vibrational systems.
Vibrations are typical and mostly unwanted phenomenon in mechanical systems, since resonance and sustained oscillations can have undesired effects such as energy waste, noise creation and even structural damage. Thus, the minimization of vibrations is a widely studied topic. There is a vast literature in this field of research, particularly in the engineering and applied mathematics. For a brief insight we give just a few references: \cite{beards1996structural, EIRivin2004, ITakewaki2009, genta2009vibration, VES2011, du2016modeling, inman2017vibration}. All these references contain the topic of the  minimization of dangerous vibrations from different aspects and for different problems.

In our setting, all matrices except $D$ are fixed and our main focus is on the damping matrix $D$. The damping matrix $D$ can be modeled in several different ways. Usually it is modeled as a sum of internal and external damping, that is, $D=D_{\mathrm{int}} + D_{\mathrm{ext}}$,  where $D_{\mathrm{int}}$ represents internal damping and the external damping part $D_{\mathrm{ext}}$ depends on positive real parameters $v_i$, $i=1,\ldots,k$ (called viscosities) and corresponding damping positions. This means that external damping encodes damping positions with corresponding damping viscosities.
The internal damping $D_{\mathrm{int}}$ can be modeled in different ways, for example as a small multiple of the critical damping or the proportional damping. Overview of different options for modeling internal damping can be found in \cite{KuzmTomljTruh12}.

We are interested in choosing a damping matrix $D$ in such a way that the vibrations of the system are as small as possible. To be able to pose this problem as an optimization problem, we need to choose an optimization criterion.

In damping optimization setting different optimization criteria are used, depending on different applications. One criterion for systems described by \eqref{eq:main} would be to minimize the $H_2$ norm of the system as the penalty function. In particular, damping optimization using the $H_2$ norm was considered in \cite{Blanchini12, ZuoNay05} as well as in \cite{TomBeatGug18, BennerKuerTomljTruh15}, where the authors considered model order reduction approaches in order to determine the optimal damping parameters efficiently. Also, one can consider damping optimization using the $H_{\infty}$ norm. This can be also calculated efficiently using the model order reduction approaches, for more details see, e.g., \cite{AliBMSV17, BenV13e}. Moreover, damping optimization in mechanical systems with external force was also considered in \cite{VES2011}, \cite{TRUTOMVES2014} and \cite{KTT16} where the authors used optimization criteria based on average energy amplitude and average displacement amplitude over the considered time. For closed systems, without external forces, there exists a multitude of optimization criteria. Overview of such criteria can be found, e.g., in  \cite{VES2011}  or \cite{NAKIC02}. Some of them are based on eigenvalues, such as spectral abscissa criterion (for more details, see, e.g., \cite{FreitasLancaster99}, \cite{TrTomljPuv18}, \cite{NAKTOMTRUH13}, \cite{WJSH2018}, \cite{GMQSW2016}), while other criteria are based on the total energy of the system, such as the total average energy. Total average energy  was considered widely in the last two decades, more details can be found in,  e.g.,  \cite{VES89}, \cite{VES90}, \cite{TRUHVES05}, \cite{TRUHVES09}, \cite{TrTomPuv17}, \cite{Cox:04} and \cite{Nakic13}. The total average energy was considered also in \cite{BennerTomljTruh10} and \cite{BennerTomljTruh11}, where the authors employed  dimension reduction techniques that allowed efficient calculation of the total average energy.

In this paper we will show that for the system given by \eqref{eq:main}, the standard $H_2$ norm may lead to an optimization problem which is not well posed in general. As the main contribution of the paper we propose an alternative criterion, which can be seen as an introduction of a constraint or alternatively as a use of a  mixed norm combining $H_2$ norm of the closed, homogeneous system with initial data and $H_2$ norm of the open, non-homogeneous systems without initial data. Moreover, using the proposed mixed norm as a criterion for the optimization, we show that the problem of global optimization problem in the case when $B_2$, $C_1$ and $C_2$ are full rank matrices has a unique solution which belongs in the class of modal damping matrices. The uniqueness of the solution and its affiliation to the class of modal matrices suggest that the new criterion should be a viable alternative to the standard approaches. We illustrate this also by some numerical examples.

The paper is organized as follows. In Section \ref{sec:h2_vibrational_system} we present a standard criterion which is based on $H_2$ norm of a vibrational system and we show its drawbacks. In Section \ref{sec:homogeneous_} we construct an $H_2$ norm for the homogeneous system while Section \ref{sec:mixed_control} considers new criterion which is based on mixed $H_2$ norms of homogeneous and non-homogeneous vibrational systems. In Section \ref{sec:numerical examples} we compare considered criteria and illustrate advantages of the new criterion as well as the optimal parameters that may arise in different cases.

\section{$H_2$ norm of a vibrational system} 
\label{sec:h2_vibrational_system}
Since the $H_2$ norm of a dynamical system will play the central role in the paper, we start with its definition.
\begin{definition}
\label{defn:h2}
	Suppose that the linear time invariant system $G$ is defined by \eqref{eq:main}.
Then the $H_2$ norm of this system, denoted by $\lVert G \rVert_2 $,
	is defined as
	\begin{equation}
		\lVert G \rVert_2 = \left( \frac{1}{2 \pi} \int_{- \infty}^{\infty} \Tr (\hat G(i \omega)^* \hat G(i \omega)) \, \mathrm{d} \omega\right)^{1/2},
	\end{equation}
	where $\hat G$ is the transfer function of the system $G$, i.e.\ after the application of the Laplace transform to \eqref{eq:main} we arrive at the equation $\hat y (s) = \hat G (s) \hat u(s)$.
\end{definition}
Note that the $H_2$ norm does not depend on the initial data $q_0$, $\dot{q}_0$, hence in the case of the $H_2$ norm we can always assume $q_0=0$, $\dot{q}_0=0$.

By $\lVert \cdot \rVert_2 $ we will denote both the $H_2$ norm of a system and the $L_2$ norm of the function. The precise meaning will always be clear from the context.

To use control--theoretic methods on \eqref{eq:main}, we can either apply Definition \ref{defn:h2} or we can linearise \eqref{eq:main} and then calculate $H_2$ norm of the corresponding linearised system. By \emph{linearization} here we mean writing the system as a first order matrix ODE. In sequel, the term linearization will always have this meaning. Linearised system is then given in the standard form $\dot{x} = A x + Bu, y = Cx$. In the following proposition we show that in both cases we obtain the same expression for the $H_2$ norm. This is a folklore result, we give a short proof for the reader's convenience.
\begin{proposition}
 	\label{pro:h2norm}
Assume that the system \eqref{eq:main} has a finite $H_2$ norm.
Then the square of the $H_2$ norm of the system \eqref{eq:main} is given by
$$\Tr \left( \begin{bmatrix}
C_1^{\ast}C_2 & 0 \\ 0 & C_2^{\ast}C_2
\end{bmatrix} X \right),$$
where $X$ is the solution of the Lyapunov equation
\begin{equation}
	\label{eq:lyap}
	\begin{bmatrix}
0 & I \\ - M^{-1} K & - M^{-1} D
\end{bmatrix} X + X \begin{bmatrix}
0 & I \\ - M^{-1} K & - M^{-1} D
\end{bmatrix}^{\ast} = - \begin{bmatrix}
	0 &0\\
0& M^{-1} B_2B_2^\ast M^{-1}
\end{bmatrix}  .
\end{equation}
 \end{proposition}
\begin{proof}
Applying the Laplace transform to \eqref{eq:main} we obtain
\[
\hat y = \hat G \hat u =
\begin{bmatrix}
C_1 (s^2 M + s D + K)^{-1} \\ C_2 s (s^2 M + s D + K)^{-1}
\end{bmatrix}
 B_2 \hat u.
\] 	
Then
\begin{align*}
\hat y &=
\begin{bmatrix}
 C_1 & 0 \\ 0 & C_2
\end{bmatrix} \begin{bmatrix}
(s^2 + s M^{-1} D + M^{-1} K)^{-1} \\ s (s^2 + s M^{-1} D + M^{-1} K)^{-1}
\end{bmatrix}
M^{-1} B_2 \hat u \\
&=
\begin{bmatrix}
 C_1 & 0 \\ 0 & C_2
\end{bmatrix} \begin{bmatrix}
\ast & (s^2 + s M^{-1} D + M^{-1} K)^{-1} \\ \ast & s (s^2 + s M^{-1} D + M^{-1} K)^{-1}
\end{bmatrix}
\begin{bmatrix}
 0 \\   M^{-1}B_2
\end{bmatrix}\hat u \\
&= \begin{bmatrix}
 C_1 & 0 \\ 0 & C_2
\end{bmatrix} \left (sI - \begin{bmatrix}
0 & I \\ - M^{-1} K & - M^{-1} D
\end{bmatrix}\right )^{-1}\begin{bmatrix}
 0 \\   M^{-1}B_2
\end{bmatrix} \hat u
.
\end{align*}
Hence the system \eqref{eq:main} has the same transfer function as the system
\begin{equation}
\label{eq:linearization0}
\left\{
\begin{aligned}
	\dot{x} &= \begin{bmatrix}
0 & I \\ - M^{-1} K & - M^{-1} D
\end{bmatrix} x + \begin{bmatrix}
 0 \\   M^{-1}B_2
\end{bmatrix}u, \\
	y &= \begin{bmatrix}
 C_1 & 0 \\ 0 & C_2
\end{bmatrix}x.
\end{aligned}
\right. 		
\end{equation}
Since the system \eqref{eq:linearization0} can be obtained from \eqref{eq:main} by the linearization $x_1 = q$, $x_2 = \dot{q}$ and the multiplication to the left of the corresponding state equation with the matrix $\diag(0,M^{-1})$,
 the statement follows from the well--known state--space formula for the $H_2$ norm (see, for example, \cite[Lemma 4.4]{zhou1998essentials}). 	
\end{proof}
\begin{remark}
	Note that by the assumption of Proposition \ref{pro:h2norm},  the system matrix $\left[\begin{smallmatrix}
0 & I \\ - M^{-1} K & - M^{-1} D
\end{smallmatrix}\right]$   is stable and hence the matrix $X$ is symmetric positive semidefinite.
\end{remark}

The choice of different linearizations of \eqref{eq:main} amounts to differnt state transformations of the system \eqref{eq:linearization0}.
Indeed, for regular matrices $T_1$  and $T_2$ let $x_1 = T_1 q$, $x_2 = T_2 \dot{q}$ be a linearization of \eqref{eq:main}. Then one calculates the corresponding matrices
\begin{equation}
 	\label{eq:abc-transf}
 A = \begin{bmatrix}
0 & T_1 T_2^{-1} \\ -T_2 M^{-1} K T_1^{-1} & -T_2 M^{-1} D T_2^{-1}
\end{bmatrix}, \,  C = \begin{bmatrix}
C_1 T_1^{-1} & 0 \\ 0 & C_2 T_2^{-1}
\end{bmatrix}, \,  B = \begin{bmatrix}
	0 \\ T_2 M^{-1} B_2
\end{bmatrix}.
\end{equation}
Now, for system matrices from \eqref{eq:linearization0} we have $ A = T \begin{bmatrix}
0 & I \\ - M^{-1} K & - M^{-1} D
\end{bmatrix} T^{-1}$,  $ B = T \begin{bmatrix}
 0 \\   M^{-1}B_2
\end{bmatrix}$ and $ C = \begin{bmatrix}
 C_1 & 0 \\ 0 & C_2
\end{bmatrix} T^{-1}$, where $T = \mathrm{diag} (T_1, T_2)$. Sometimes we will write $A=A(D)$ and $G=G(D)$ to denote that the matrix $A$ and the system $G$ depend on $D$.  Thus, in general, our linear time invariant system can be written as

\begin{equation}
\label{eq:linearization}
\left\{
\begin{aligned}
	\dot{x} &=  A x +   Bu, \\
	y &=   C x.
\end{aligned}
\right. 		
\end{equation}

Hence, using \cite[Lemma 4.4]{zhou1998essentials}, the $H_2$ norm of the system \eqref{eq:main} can be calculated as
\begin{equation}
	\label{eq:H2_norm}
	\sqrt{\Tr (  C^{\ast} C X )} \text{ where }  A X + X  A^{\ast} = -  B  B^{\ast}.
\end{equation}

In the next proposition we show that the optimization problem $\min_D \lVert G(D) \rVert_2 $ does not in general admit a solution. We treat a special case of a vibrational system without gyroscopic forces and assume that the damping is passive, so no additional energy is introduced into the system. These assumptions imply that $M$ and $K$ are positive definite symmetric matrices and that matrices $D$ are positive semidefinite symmetric matrices. We additionally assume that the matrices $A(D)$ are stable, so that the Lyapunov equation from \eqref{eq:H2_norm} always has a solution.

In the sequel, the relation $X\ge 0$ will be used to denote the fact that the symmetric matrix $X$ is positive semidefinite.

\begin{proposition}
	\label{pro:h2notgood}
Assume that $M$ and $K$ are symmetric positive definite matrices. Assume also that
$C_2 B_2 \ne 0$. 
Then the optimization problem
\[
\min \{\lVert G( D) \rVert_2\colon D \text{ positive semidefinite and } A(D)\text{ a stable matrix}\} 	
\]	
does not have a solution and moreover 
\[ \inf \{\lVert G( D) \rVert_2\colon D \text{ positive semidefinite and } A(D)\text{ a stable matrix}\} = 0. \]
\end{proposition}
\begin{proof}
Since the choice of linearization does not change the $H_2$ norm, we will work with a particulary convenient linearization given in \eqref{eq:abc_proper} below.
Let $M=L_2L_2^*$ and $K=L_1L_1^*$ be Cholesky factorizations of $M$ and $K$ and let $L_2^{-1}L_1 = U_2 \Omega U_1^*$ be a SVD decomposition of $L_2^{-1}L_1$. Note that the diagonal elements of $\Omega = \mathrm{diag}(\omega_1, \ldots, \omega_n)$ are square roots of the eigen--frequencies of the corresponding undamped system (with $D=0$).
We choose $\widetilde T  = \mathrm{diag}(\widetilde T_1, \widetilde T_2) = \mathrm{diag} (U_1^* L_1^*, U_2^* L_2^*)$. Then we obtain for the corresponding system matrices
\begin{equation}
	\label{eq:abc_proper}
\widetilde A = \begin{bmatrix}
0 & \Omega \\ -\Omega & - \widetilde D
\end{bmatrix}, \quad \widetilde C = \begin{bmatrix}
\widetilde C_1 & 0 \\ 0 & \widetilde C_2
\end{bmatrix}, \quad  \widetilde B = \begin{bmatrix}
	0 \\ \widetilde B_2
\end{bmatrix},	
\end{equation}
where $\widetilde D = U_2^* L_2^{-1} D L_2^{-*}U_2$, $\widetilde C_1 = C_1 L_1^{-*}U_1$, $\widetilde C_2 = C_2 L_2^{-*} U_2$ and $\widetilde B_2 = U_2^* L_2^{-1} B_2$.

We will show that for the damping matrices $\widetilde D (\alpha) = \alpha \Omega$ the $H_2$ norm of the corresponding system $G(\alpha):= G(\widetilde D(\alpha))$ tends to zero when $\alpha \to \infty$. First note that $\widetilde A(\widetilde D(\alpha))$ is stable since $\widetilde D(\alpha)$ is positive definite (see, for example, \cite[Corollary 15.7]{VES2011}). Let us partition the solution of the Lyapunov equation $\widetilde A (\widetilde D (\alpha)) X + X \widetilde A (\widetilde D (\alpha)^* = - \widetilde B \widetilde B^*$ as $2\times 2$  block matrix with $n\times n$ entries $X_{ij}(\alpha)$, which depend on $\alpha$. Then we have
\begin{gather}
	\label{eq:lyap_eq_1}
 	\Omega X_{12}(\alpha)^{\ast} + X_{12}(\alpha)\Omega = 0 \\
 	\label{eq:lyap_eq_2}
 	\Omega X_{22}(\alpha) - X_{11}(\alpha)\Omega - \alpha X_{12}(\alpha)\Omega = 0  \\
 	\label{eq:lyap_eq_3}
 	- \Omega X_{12}(\alpha) - \alpha \Omega X_{22}(\alpha) - X_{12}^{\ast}(\alpha)\Omega - \alpha X_{22}(\alpha)\Omega = - \widetilde B \widetilde B^{\ast}.
 \end{gather}
First we assume that $\alpha^{-1} X_{12}(\alpha)\to 0$ as $\alpha\to \infty$. Then \eqref{eq:lyap_eq_3} reads
\[
-  \Omega X_{22}(\alpha) -  X_{22}(\alpha)\Omega = -\frac{1}{\alpha} \widetilde B \widetilde B^{\ast} + \frac{1}{\alpha} \Omega X_{12}(\alpha) + \frac{1}{\alpha} X_{12}^{\ast}(\alpha)\Omega \to 0,
\]
hence $X_{22}(\alpha) \to 0$ as $\alpha\to \infty$. Since $X(\alpha)\ge 0$ it follows $X_{12}(\alpha)\to 0$ as $\alpha\to\infty$. From \eqref{eq:lyap_eq_1} it follows that $\Tr (X_{12}(\alpha) \Omega) = 0$, and from \eqref{eq:lyap_eq_2}
\[
\Tr (X_{11}(\alpha)\Omega) = \Tr(\Omega X_{22}(\alpha)) - \alpha \Tr(X_{12}(\alpha)\Omega) = \Tr(\Omega X_{22}(\alpha))\to 0 \text{ as } \alpha\to \infty.
\]
But this implies $X_{11}(\alpha)\to 0 $, hence $X(\alpha)\to 0 $ as $\alpha\to\infty$. Then $\lVert G(\alpha) \rVert_2 \to 0 $ as $\alpha\to\infty$.

In the other case, when $\alpha^{-1} X_{12}(\alpha)\nrightarrow 0$ as $\alpha\to \infty$, let $k\in \mathbb{N}$ be the smallest integer such that $\alpha^{-k} X_{12}(\alpha)\to 0$ as $\alpha\to \infty$. Such $k$ exists since $X(\alpha)$ is a rational function of $\alpha$, which follows from \eqref{eq:H2_norm}. Since $\alpha^{-k+1} X_{12}(\alpha)\nrightarrow 0$ as $\alpha\to \infty$, the non--negativity of $X(\alpha)$ implies that $\alpha^{-k+1} X_{22}(\alpha) \nrightarrow 0 $ as  $\alpha\to \infty$.
Dividing \eqref{eq:lyap_eq_3} with $\alpha^k$ we obtain
\[
- \frac{1}{\alpha^{k-1}} \Omega X_{22}(\alpha) - \frac{1}{\alpha^{k-1}} X_{22}(\alpha)\Omega = - \frac{1}{\alpha^{k}} \widetilde B \widetilde B^{\ast} + \frac{1}{\alpha^{k}} \Omega X_{12}(\alpha) + \frac{1}{\alpha^{k}} X_{12}^{\ast}(\alpha)\Omega \to 0,
\]
as $\alpha\to \infty$. But this implies $\alpha^{-k+1}X_{22}(\alpha)\to 0$ as $\alpha\to \infty$, a contradiction. To finish the proof, note that $\lVert G(D) \rVert_2 = 0$ if and only if $G(D) = 0$ which implies that $CA^k B = 0$, $k=0,1,\ldots$ (see, for example, \cite[Lemma 2.26]{dullerud2013}). This implies $C_2 B_2 = 0$, a contradiction to our assumption $C_2 B_2 \ne 0$.
\end{proof}
\begin{remark}
	One way to look at Proposition \ref{pro:h2notgood} is to note that
	one can interpret the $H_2$ norm of the system as the square root of the sum of the norms of the system responses to the initial data of the form $x(0)= B e_i$ with zero external force, where $e_i$ denotes the $i$th canonical vector. Indeed, in that case the corresponding system responses are $y_i(t) = C\mathrm{e}^{At}B e_i$, so
	\begin{align*}
		\sum_{k=1}^{m} \lVert y_i \rVert_2^2 &= \sum_{k=1}^{m} \int_0^{\infty}  e_i^* B^* \mathrm{e}^{A^* t} C^* C \mathrm{e}^{At} B e_i \deri t = \sum_{k=1}^m \Tr \left(e_i^* B^* \int_0^{\infty} \mathrm{e}^{A^*t}C^* C^* \mathrm{e}^{At} \deri t B e_i\right)\\
		&= \sum_{k=1}^m \Tr (e_i^* B^* X B e_i) = \Tr (B^* X B) = \lVert G \rVert_2^2 ,
	\end{align*}
	where $X$ solves the Lyapunov equation $A^* X + X A = - C^* C$.
	Hence in our case, the $H_2$ norm only measures the responses to initial data only consisting of velocities, initial displacements do not play  any role.
	This seems to be a general issue when the (first order) control system is obtain by a linearization from the higher order systems.
	\end{remark}
\begin{remark}
The assumption $C_2B_2\ne 0$ could be relaxed, the system \eqref{eq:main} is zero-system under a much stronger condition than $C_2B_2 = 0$. Since the precise statement and its proof are complicated, we will not formulate it here.
\end{remark}


\section{$H_2$ norm of a homogeneous system} 
\label{sec:homogeneous_}
In this section we will generalize the total energy approach for the measurement of unwanted vibrations of a homogeneous vibrational system which is in a way counterpart to the $H_2$ norm of the system given in preceding section.

For the system given by \eqref{eq:main} we take $u=0$ but now the initial conditions $q(0)=q_0$, $\dot{q}(0)=\dot{q}_0$ will play a role. With $e(t;q_0,\dot{q}_0) = \lVert y(t;q_0,\dot{q}_0) \rVert^2$ we denote the energy of the output of the system. We want to average the total energy of the output, given by $\int_0^{\infty} e(t) \, \mathrm{d}t$ over all initial data $(q_0,\dot{q}_0)$ on the corresponding unit sphere.

We calculate, using a linearization given by $T = \mathrm{diag}(T_1,T_2)$,
\begin{equation*}
	e(t) = \left\lVert \begin{bmatrix}
	C_1 q(t) \\ C_2 \dot{q}(t)
	\end{bmatrix}  \right\rVert^2  =
	\left\lVert  C \begin{bmatrix}
	x_1(t) \\ x_2(t)
	\end{bmatrix} \right\rVert^2 =
	\left\lVert  C \mathrm{e}^{ A t} T
	\begin{bmatrix}
	q_0 \\ \dot{q}_0
	\end{bmatrix}
	 \right\rVert^2
	= \begin{bmatrix}
	q_0 \\ \dot{q}_0
	\end{bmatrix}^* T^* \mathrm{e}^{ A^* t}  C^*  C
	\mathrm{e}^{ A t}
	T
	\begin{bmatrix}
	q_0 \\ \dot{q}_0
	\end{bmatrix}.
\end{equation*}
Now
\[
\int_0^{\infty} e(t) \, \mathrm{d}t = \begin{bmatrix}
	q_0 \\ \dot{q}_0
	\end{bmatrix}^* T^* \int_0^{\infty} \mathrm{e}^{ A^* t}  C^*  C
	\mathrm{e}^{ A t} \, \mathrm{d}t T
	\begin{bmatrix}
	q_0 \\ \dot{q}_0
	\end{bmatrix} = \begin{bmatrix}
	q_0 \\ \dot{q}_0
	\end{bmatrix}^* T^* X T
	\begin{bmatrix}
	q_0 \\ \dot{q}_0
	\end{bmatrix},
\]
where $X$ is the solution of the Lyapunov equation
\begin{equation}
\label{eq:lyap_meas}
 A^* X + X  A = -  C^*  C.	
\end{equation}
Multiplying this equation on the left with $T^*$ and on the right with $T$ we obtain
\[
(T^{-1} A T)^* (T^* X T) + (T^* X T) (T^{-1} A T) = - \begin{bmatrix}
C_1^* C_1 & 0 \\ 0 & C_2^* C_2
\end{bmatrix}.
\]
Let $Y = T^* X T$. Then $Y$ is the solution of the Lyapunov equation
\begin{equation*}
	\label{eq:lyap_meas_pom}
	(T^{-1} A T)^* Y + Y (T^{-1} A T)^* = - \begin{bmatrix}
C_1^* C_1 & 0 \\ 0 & C_2^* C_2
\end{bmatrix}.
\end{equation*}
Let us choose $\sigma$, a surface measure on the unit sphere $\mathbb{R}^{2n}$. The $H_2$ norm of the homogeneous system is defined by
\[
\lVert G \rVert_{2, \mathrm{hom}}^2 :=  \int_{\lVert q_0 \rVert^2  + \lVert \dot{q}_0 \rVert^2  = 1 } \int_0^{\infty} e(t;q_0,\dot{q}_0) \, \mathrm{d}
t \, \mathrm{d} \sigma = \int_{\lVert q_0 \rVert^2  + \lVert \dot{q}_0 \rVert^2  = 1 } \begin{bmatrix}
q_0 \\ \dot{q_0}
\end{bmatrix}^* Y \begin{bmatrix}
q_0 \\ \dot{q_0}
\end{bmatrix} \, \mathrm{d} \sigma.
\]
Since
\[
Y \mapsto \int_{\lVert q_0 \rVert^2  + \lVert \dot{q}_0 \rVert^2  = 1 } \begin{bmatrix}
q_0 \\ \dot{q_0}
\end{bmatrix}^* Y \begin{bmatrix}
q_0 \\ \dot{q_0}
\end{bmatrix} \, \mathrm{d} \sigma
\]
is a linear functional on the space of symmetric matrices, by the Riesz representation theorem there exists a symmetric positive semidefinite matrix $\hat Z_\sigma$ (\cite[Proposition 21.1]{VES2011}) such that
\[
\int_{\lVert q_0 \rVert^2  + \lVert \dot{q}_0 \rVert^2  = 1 } \begin{bmatrix}
q_0 \\ \dot{q_0}
\end{bmatrix}^* Y \begin{bmatrix}
q_0 \\ \dot{q_0}
\end{bmatrix} \, \mathrm{d} \sigma = \Tr (Y\hat Z_\sigma) = \Tr (T^* X T \hat Z_\sigma) = \Tr (T \hat Z_\sigma T^* X) .
\]
Note that if we chose the Lebesgue measure for $\sigma$, then $\hat Z_\sigma = \frac{1}{2n}I$. Here by the Lebesgue measure on the unit sphere we mean the surface measure obtained by the Minkowski formula taking Lebesgue measure as the ambient measure, for the comprehensive treatment see \cite{federer1969}. A formula for the matrix $\hat Z_\sigma$ when the measure $\sigma$ is Gaussian is given in \cite{nakic2013integration}.

Let us denote $ Z_\sigma = T \hat Z_\sigma T^*$. Then succinctly the $H_2$ norm of the homogeneous system can be calculated as
$\lVert G \rVert_{2, \mathrm{hom}} = \sqrt{\Tr ( Z_\sigma X)}$, where $X$ solves \eqref{eq:lyap_meas}. In the sequel it will be useful to calculate the same norm by the use of the dual Lyapunov equation:
\begin{equation}
	\label{eq:H2_norm_homogeneous}
	\sqrt{\Tr ( C^*  C Y)} \text{ where }  A Y + Y  A^* = - Z_\sigma.
\end{equation}
Note that \eqref{eq:H2_norm} and \eqref{eq:H2_norm_homogeneous} differ only in the right hand side of the Lyapunov equations, where right hand sides encode the information about dangerous external forces and initial conditions, respectively.


\section{Mixed control of vibrational systems} 
\label{sec:mixed_control}

In this section we will combine the two approaches given in Sections \ref{sec:h2_vibrational_system} and \ref{sec:homogeneous_} to create a norm which takes into account both external forces and initial data and we will show
that using this norm as a criterion for the optimization we can prove that the global minimum exists and is obtained in the class of modal damping matrices.

The issue with \eqref{eq:H2_norm_homogeneous} is that it does not carry any information about the external forces, and the issue with \eqref{eq:H2_norm} is that it does not carry all the needed information about the initial data. A natural choice is to try to combine these two norms by taking their (generalized) quadratic mean. Let $0<p<1$. We define the \emph{$p$--mixed $H_2$ norm} of the system $G$ by $\lVert G \rVert_{2,p}^2 = (1-p)\lVert G \rVert_2^2 + p \lVert G \rVert_{2,\mathrm{hom}}^2$. Using \eqref{eq:H2_norm} and  \eqref{eq:H2_norm_homogeneous}, we see that the $p$--mixed $H_2$ norm can be calculated as
\begin{equation}
	\label{eq:mixed_norm}
	\sqrt{\Tr ( C^*  C X)}, \text{ where }  A X + X  A^* = -p  Z_\sigma - (1-p)  B  B^*.
\end{equation}
Note that \eqref{eq:mixed_norm} does not depend on the choice of the linearization. One can also think of \eqref{eq:mixed_norm} as the standard $H_2$ norm with an additional geometric constraint. Indeed, we can think of the function $G\mapsto p \lVert G \rVert_{2,\mathrm{hom}}^2$ as a barrier function, constraining the set of feasible systems $G$.

Let us now assume that $M$ and $K$ are positive definite symmetric matrices. When modeling vibrational systems, the matrix $B_2$ is usually designed as a band--pass filter where only the dangerous frequencies are passed. In the linearization given in \eqref{eq:abc_proper}, this would mean that $\widetilde B_2 = Z_1 := \mathrm{diag}(I_r,0_{n-r})$, where we assumed that we chose SVD decomposition $L_2^{-1}L_1 = U_2 \Omega U_1^*$ in such a way that the dangerous frequencies of the undamped system are exactly $\omega_1, \ldots, \omega_r$. Hence in that case $B_2 = L_2 U_2 Z_1$.

The measure $\sigma$ typically is chosen in such a way that it attenuates frequencies which are not dangerous. In particular, if the surface measure is chosen in such a way that it corresponds to Lebesgue measure on the subspace spanned by the vectors $[x_i, 0]^{\T}$ and $[0,x_i]^{\T}$, $i=1,\ldots,r$, where $x_i$ are the eigenvectors of the first $r$ undamped eigenfrequencies and on the rest of $\mathbb{R}^{2n}$ it corresponds to the Dirac measure concentrated at zero, then, in the linearization given by \eqref{eq:abc_proper}, we have $\widetilde Z_\sigma = \frac{1}{2n}Z$, where $Z = \mathrm{diag}(Z_1,Z_1)$, $Z_1 = \mathrm{diag}(I_r,0_{n-r})$.

Since the output $y$ usually corresponds to the energy corresponding to the unwanted vibrations, a usual choice for matrices $C_1$ and $C_2$ is $C_1 = \frac{1}{\sqrt{2}} L_1^*$, $C_2 = \frac{1}{\sqrt{2}} L_2^*$. Then $\lVert y(t) \rVert^2 $ equals the energy of the system at the time $t$. This corresponds to $\widetilde C_1 = \frac{1}{\sqrt{2}} U_1$, $\widetilde C_2 = \frac{1}{\sqrt{2}} U_2$ and hence $\widetilde C^* \widetilde C = \frac{1}{2}I $. If we are (as is usually the case) interested only in the energy output corresponding to a part of the system, then  we have $\widetilde C^* \widetilde C = \frac{1}{2}Z $. But in some cases there would be other appropriate choices for $C_1$ and $C_2$. One natural setting is when the matrix $C$ is such that $\frac{\mathrm{d}}{\mathrm{d}t} \lVert y(t) \rVert \le 0$, which is equivalent to the assumption that $C^*CA$ is a stable matrix. Note that this is not automatically satisfied as the following example shows.
\begin{example}
Let
\[
S = \begin{bmatrix}
4 & 0 \\ 0 & 1
\end{bmatrix}, \,
A = \begin{bmatrix}
1 & 1 \\ -4 & -3
\end{bmatrix}.
\]	
Then $A$ is stable, but $SA$ is not.
\end{example}
From the practical point of view, it is important to note that $C$ usually will have low rank.

One typical situation is that $B_2$, $C$ and the measure $\sigma$ are such that
\begin{equation}\label{B2Cmatrices}
    \widetilde B_2 = Z_1 ,\quad  \widetilde C^* \widetilde C = \frac{1}{2}Z\quad  \mbox{and}\quad  \widetilde Z_\sigma = \frac{1}{2n}Z.
\end{equation}
If we drop the constants which are not relevant for the optimization purposes, we obtain
\begin{equation}
	\label{eq:mixed_norm_final}
	\sqrt{Tr (ZX)}, \text{ where } \widetilde A X + X \widetilde A^* = - \left[ \begin{smallmatrix}
		p Z_1 & 0 \\ 0 & Z_1
		\end{smallmatrix} \right].
\end{equation}
The next theorem shows that one can calculate the global minimum in this particular case.
\begin{theorem}
\label{thm:glo-opt}
Assume that $M$ and $K$ are symmetric positive definite matrices. Let $Z_1 = I$ in \eqref{eq:mixed_norm_final}.
Let
\[ \mathcal{D}_s=\{\widetilde D\in \mathbb{R}^{n\times n}: \widetilde D \ge 0 \text{ and the corresponding } \widetilde A \text{ is stable} \}. \]
Then for all $0<p<1$ there exists a unique global minimum of the following optimization problem:
\[
\text{minimize } \sqrt{\Tr (X)} \text{ subject to }\widetilde A X + X \widetilde A^* = - \left[ \begin{smallmatrix}
		p I & 0 \\ 0 & I
		\end{smallmatrix} \right] \text{ and } \widetilde D \in \mathcal{D}_s.		
\]
The minimum is attained at $\widetilde D = \sqrt{\frac{2(1+p)}{p}}\Omega$.
\end{theorem}
\begin{proof}
To emphasize the dependence of $X$ to the parameter $\widetilde D$ we will sometimes write $X(\widetilde D)$.

Let $\widetilde D \in \mathcal{D}_s$ be arbitrary. From \cite[Corollary 15.7]{VES2011} it follows that $\widetilde D \in \mathcal{D}_s$ if and only if all diagonal entries of $\widetilde D$ are non-zero.
Let $ Z_{(i)}$ be a diagonal matrix with all diagonal entries zero except the $i$-th which is $1$. Set
$Z^i=\left(\begin{smallmatrix}
p  Z_{(i)} & 0 \\ 0 &  Z_{(i)}
\end{smallmatrix}\right)$. Let $X_i$ be the solution of the Lyapunov equation
\begin{equation}
\label{eq:lyap_i}
 \widetilde A (\widetilde D)X+X \widetilde A ( \widetilde D)^{\ast}=-Z^i.
\end{equation}
Then it is easy to see that the solution of the Lyapunov equation
$\widetilde A (\widetilde D)X+X \widetilde A ( \widetilde D)^{\ast}=- \left[ \begin{smallmatrix}
		p I & 0 \\ 0 & I
		\end{smallmatrix}\right]$
is
\begin{equation}
\label{eq:sumxi}
X=\sum_{i=1}^n X_i.
\end{equation}
Our aim is to show
\begin{equation}
\label{eq:mintr}
 \min \{ \Tr(X): X \; \text{solves (\ref{eq:lyap_i})},\; \widetilde D\in D_s\} \ge \frac{\sqrt{2 p (1 + p)}}{\omega_i},
\;\; i=1,\ldots,n.
\end{equation}
Here the right hand side is chosen in such a way that when we sum it for $i=1,\ldots,n$, we obtain the minimal value of the optimization problem.
Observe that by simple permutation argument we can assume
$i=1$.
Let us decompose a matrix $X_1\in\mathbb{R}^{2n\times 2n}$ in the following way:
\begin{equation}
\label{eq:X1}
 X_1=\begin{bmatrix}
x_{11} & X_{12} & x_{13} & X_{14} \\
X_{12}^{\ast} & X_{22} & X_{23} & X_{24} \\
x_{13} & X_{23}^{\ast} & x_{33} & X_{34} \\
X_{14}^{\ast} & X_{24}^{\ast} & X_{34}^{\ast} & X_{44}
\end{bmatrix},
\end{equation}
where $x_{11},x_{33}, x_{13}\in\mathbb{R}$,
$X_{12},X_{14},X_{34}\in\mathbb{R}^{1\times (n-1)}$,
$X_{22},X_{24},X_{44}\in \mathbb{R}^{(n-1)\times(n-1)}$, and
$X_{23}\in\mathbb{R}^{(n-1)\times 1}$. Next we partition the Lyapunov
equation
\[ \widetilde A (\widetilde D)X_1+X_1\widetilde A (\widetilde D)^{\ast}=-Z^1 \]
in the same way as we did the matrix $X_1$. We obtain
\begin{gather}
x_{13}\omega_1 + \omega_1x_{13}^{\ast} + p=0     \tag*{(1,1)}   \\
\omega_1X_{23}^{\ast} + X_{14}\Omega_{n-1} = 0   \tag*{(1,2)}    \\
\omega_1x_{33} - x_{11}\omega_1 - x_{13}\widetilde d_{11} - X_{14}\widetilde D_{12}^{\ast} = 0   \tag*{(1,3)}  \\
\omega_1X_{34} - X_{12}\Omega_{n-1} - x_{13}\widetilde D_{12} - X_{14}\widetilde D_{22} = 0    \tag*{(1,4)}
\end{gather}
\begin{gather}
\Omega_{n-1}X_{24}^{\ast} + X_{24}\Omega_{n-1} = 0    \tag*{(2,2)}   \\
\Omega_{n-1}X_{34}^{\ast} - X_{12}^{\ast}\omega_1 - X_{23}\widetilde d_{11} - X_{24}\widetilde D_{12}^{\ast} = 0   \tag*{(2,3)}   \\
\Omega_{n-1}X_{44} - X_{22}\Omega_{n-1} - X_{23}\widetilde D_{12} - X_{24}\widetilde D_{22} = 0   \tag*{(2,4)}    \\
-\omega_1x_{13} - \widetilde d_{11}x_{33} - \widetilde D_{12}X_{34}^{\ast} - x_{13}^{\ast}\omega_1 - x_{33}\widetilde d_{11} - X_{34}\widetilde D_{12}^{\ast} + 1 = 0 \tag*{(3,3)} \\
-\omega_1 X_{14} - \widetilde d_{11}X_{34} - \widetilde D_{12}X_{44} - X_{23}^{\ast}\Omega_{n-1} - x_{33}\widetilde D_{12} - X_{34}\widetilde D_{22} = 0 \tag*{(3,4)} \\
-\Omega_{n-1}X_{24} - \widetilde D_{12}^{\ast}X_{34} - \widetilde D_{22}X_{44} - X_{24}^{\ast}\Omega_{n-1} - X_{34}^{\ast}\widetilde D_{12} - X_{44}\widetilde D_{22} = 0, \tag*{(4,4)}
\end{gather}
where $\omega_1,\widetilde d_{11}\in \mathbb{R}$, $\widetilde D_{12}\in \mathbb{R}^{1\times (n-1)}$, and
 $\widetilde D_{22}, \Omega_{n-1}\in\mathbb{R}^{(n-1)\times (n-1)}$.

From (1,1) we obtain $x_{13}=-\frac{p}{2\omega_1}$.
Since $\widetilde D\ge 0$, one can easily see that $\widetilde d_{11}=0$ implies $\widetilde D_{12}=0$, hence (3,3) reads $p=-1$, a contradiction. Hence, $\widetilde d_{11}>0$.
From (3,3) we now get
\begin{equation}
\label{eq:trx33}
 x_{33}=\frac{1 + p - 2 X_{34}\widetilde D_{12}^{\ast}}{2\widetilde d_{11}}.
\end{equation}
The relation (4,4), together with the facts $X_{44}\ge 0$, $\widetilde D_{22}\ge 0$,  implies
\[ \Tr(\widetilde D_{12}^{\ast}X_{34}+X_{34}^{\ast}\widetilde D_{12})\le  -\Tr(\Omega_{n-1}X_{24}+X_{24}^{\ast}\Omega_{n-1}), \]
and the relation (2,2) implies $\Tr(X_{24}\Omega_{n-1})=0$, hence we obtain
\begin{equation}
\label{eq:x34c12}
 \Tr(X_{34}^{\ast}\widetilde D_{12})=\Tr(X_{34}\widetilde D_{12}^{\ast})\le 0.
\end{equation}
From the relation (1,3) we obtain
\[ x_{11}=x_{33}-x_{13}\widetilde d_{11}\omega_1^{-1}-\omega_1^{-1}X_{14}\widetilde D_{12}^{\ast}. \]
From relation (2,4) we obtain
\[ X_{22}=\Omega_{n-1}X_{44}\Omega_{n-1}^{-1}-X_{23}\widetilde D_{12}\Omega_{n-1}^{-1}-X_{24}\widetilde D_{22}\Omega_{n-1}^{-1}, \]
hence
\[ \Tr X_{22}= \Tr X_{44}- \Tr(X_{23}\widetilde D_{12}\Omega_{n-1}^{-1})-
\Tr(X_{24}\widetilde D_{22}\Omega_{n-1}^{-1}). \]
From the relation (2,2) we obtain
\[ X_{24}=\frac{1}{2}S\Omega_{n-1}^{-1}, \]
where $S\in\mathbb{R}^{(n-1)\times (n-1)}$ is a skew--symmetric matrix.

Hence
\begin{equation*}
\begin{split}
\Tr X_1&=x_{11}+\Tr X_{22}+x_{33}+\Tr X_{44}\\
&= 2x_{33}+2\Tr X_{44}+\frac{p \widetilde d_{11}}{2\omega_1^2}-\frac{1}{\omega_1}X_{14}\widetilde D_{12}^{\ast}-
\Tr(X_{23}\widetilde D_{12}\Omega_{n-1}^{-1})-\frac{1}{2}\Tr(S\Omega_{n-1}^{-1}\widetilde D_{22}\Omega_{n-1}^{-1})\\
&=2x_{33}+2\Tr X_{44}+\frac{p \widetilde d_{11}}{2\omega_1^2}-\frac{1}{\omega_1}X_{14}\widetilde D_{12}^{\ast}-
\Tr(X_{23}\widetilde D_{12}\Omega_{n-1}^{-1}).
\end{split}
\end{equation*}
From the relation (1,2) follows $X_{23}=-\frac{1}{\omega_1}\Omega_{n-1}X_{14}^{\ast}$,
hence
\[ \Tr X_1=2x_{33}+2\Tr X_{44}+\frac{p \widetilde d_{11}}{2\omega_1^2}. \]
Now (\ref{eq:trx33}) and (\ref{eq:x34c12}) imply
\begin{equation}
\label{eq:nejedx}
\Tr X_1= \frac{1 + p - 2X_{34}\widetilde D_{12}^{\ast}}{\widetilde d_{11}}+\frac{p \widetilde d_{11}}{2\omega_1^2}+2\Tr X_{44}
\ge \frac{1 + p}{\widetilde d_{11}}+\frac{p \widetilde d_{11}}{2\omega_1^2}\ge \frac{\sqrt{2 p (1 + p)}}{\omega_1}.
\end{equation}
The last inequality follows from the following observation.
Let us define the function $g(x)=\frac{1+p}{x}+\frac{p x}{2\omega_1^2}$. Then the function $g$ attains its unique minimum
$\frac{\sqrt{2 p (1 + p)}}{\omega_1}$ in $x=\sqrt{\frac{2(1+p)}{p}}\omega_1$.

Hence, we have shown (\ref{eq:mintr}). Now (\ref{eq:sumxi}) implies
\[ \Tr(X(\widetilde D))\ge \sqrt{2 p (1 + p)}\sum_{i=1}^n \omega_i^{-1}. \]
Since
\[ \Tr\left(X\left(\sqrt{\frac{2(1+p)}{p}}\Omega\right)\right)=\sqrt{2 p (1 + p)}\sum_{i=1}^n \omega_i^{-1},\]
this is indeed the global minimum.

Assume that $\widetilde D\in D_s$ is such that
$\Tr(X(\widetilde D))=\sqrt{2 p (1 + p)}\sum_{i=1}^n \omega_i^{-1}$. Then
(\ref{eq:nejedx}) and (\ref{eq:sumxi}) imply $\Tr
X_1=\sqrt{2 p (1 + p)}\omega_1^{-1}$. Observe that the matrix $X_1$ is decomposed as in (\ref{eq:X1}). Then (\ref{eq:nejedx}) implies
$X_{44}=0$. Since $X_1\ge 0$, it follows $X_{14}=X_{24}=X_{34}=0$.
From the relation (1,2) it follows $X_{23}=0$, from relation (2,4) it follows $X_{22}=0$, and from relation (2,3) it follows $X_{12}=0$. Finally, from (1,4) now it follows $\widetilde D_{12}=0$.

By repeating this procedure for $i=2,\ldots,n$ we obtain that $\widetilde D$ is a diagonal matrix. From \eqref{eq:nejedx}
it follows that $\widetilde D = \alpha \Omega^{-1}$ for some $\alpha$ and then it is easy to check that $\alpha = \sqrt{2(1+p)}/\sqrt{p}$.
\end{proof}
\begin{remark}
	Theorem \ref{thm:glo-opt} shows that the $p$--mixed $H_2$ norm should be well suited for the use in damping optimization problems for vibrational systems, at least in the case when there are no gyroscopic forces. The unique global minimizer is obtained in the class of the modal matrices, which is a very natural result.
\end{remark}

\section{Numerical experiments}
\label{sec:numerical examples}

In this section, we consider numerical examples in order to illustrate the quality and advantages of new performance measures as well as the comparison with standard performance measures. In these examples the corresponding Lyapunov equation was solved by Matlab's  function \verb+lyap+.

\begin{example}\label{ex1}
{\em We consider an $n$-mass oscillator  or an oscillator ladder given by \ref{n-massOscillator}. The oscillator describes the mechanical system of $n$ masses and $n+1$ springs with
two grounded dampers. Similar models were considered e.g. in  \cite{BennerTomljTruh10}, \cite{BennerTomljTruh11}, \cite{TrTomljPuv18} and \cite{VES2011}.
For this  mechanical system, the mathematical model is given by
(\ref{eq:main}), where the mass and stiffness matrices are defined by
\begin{align*}
M&=\diag(m_1,m_2,\ldots,m_n),\\
K&=\left(%
\begin{array}{ccccc}
  k_1+k_2 & -k_2 &  &  &  \\
  -k_2    & k_2+k_3& -k_3 &  &  \\
          & \ddots   & \ddots & \ddots &  \\
          &      & -k_{n-1} & k_{n-1}+k_{n} & -k_n   \\
          &      &  & -k_n & k_n+k_{n+1} \\
\end{array}%
\right).
\end{align*}
Coefficients of mass and stiffness matrices are given as
\begin{align*}
 n&=100;  \\
    k_i&=100, \quad\forall i;
  & m_i= \left\{ \begin{array}{ll} 200- 2i, \quad & i=1,\ldots, 50, \\
   i+50 , \quad & i=51,\ldots,100.
                 \end{array}\right.
\end{align*}

\begin{figure}[h!]
\centering
\includegraphics[width=10cm]{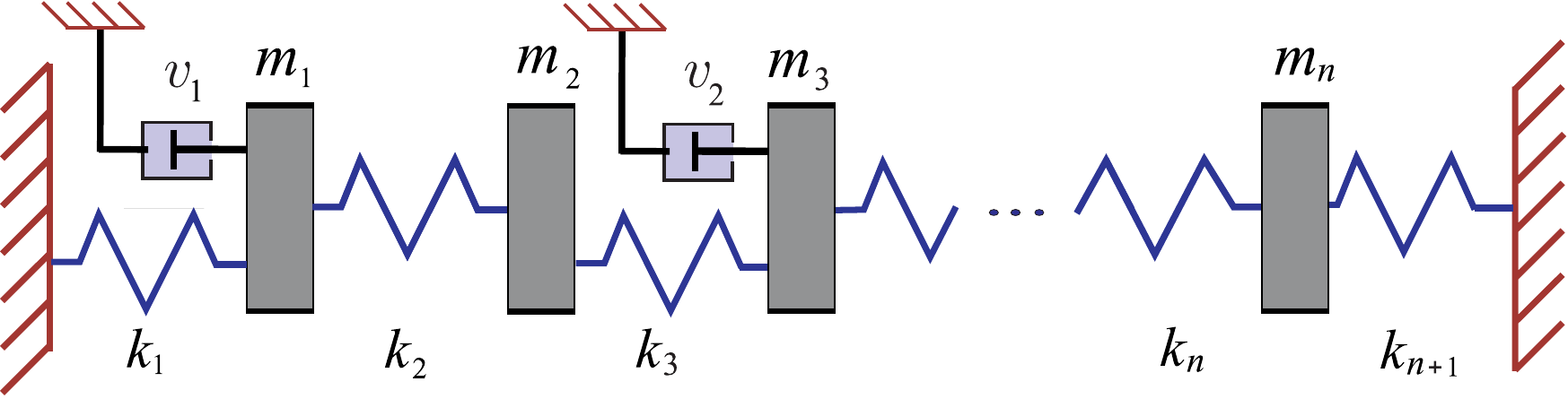}
 \caption{The $n$-mass oscillator with two grounded dampers}
 \label{n-massOscillator}
\end{figure}

The damping matrix is given as
$ D =D_{\mathrm{int}} + D_{\mathrm{ext}},$  where the  internal damping $D_{int}$ is small multiple of the critical damping, $D_{\mathrm{int}} =0.04 \cdot M^{1/2}\sqrt{M^{-1/2}KM^{-1/2}}M^{1/2}$. Here $\sqrt{X}$ denotes the unique symmetric positive definite matrix such that $(\sqrt{X})^2 = X$.

We  consider two grounded dampers with the positions $i$ and $j$ that have the viscosities $v_1,v_2\geq 0 $, respectively. That is,
the external damping is defined by $D_{ext}=v_1 e_{i}e_{i}^T+v_2 e_{j}e_{j}^T$, where $1\leq i<j\leq n$.  We are interested in the damping of all undamped eigenfrequencies.

In order to compare optimal positions when considering different performance measures, we will calculate the optimal viscosities for all possible damping configurations where feasible interval for  each viscosity will be $[\,0,\,5000 \,]$. The viscosities were optimized by Matlab's function \verb+fminsearchbnd+ and they are rounded to two digits.

Regarding the input and  the output matrices we consider two cases. The first case considers one particular choice of the input and output matrices, while the second case considers matrices given by \eqref{B2Cmatrices}. In both cases we will calculate new performance measure using the linearization given by \eqref{eq:abc_proper}.

Let us consider the first case where for the penalty function we will use \emph{$p$--mixed $H_2$ norm} of the system defined by \eqref{eq:mixed_norm}.

Here the input matrix $B_2\in \mathbb{R}^{n\times 5}$ from \eqref{eq:main} is such that
 \begin{align*}
  B_2(1:5,1:5) &=\diag(5,4,3,2,1),
 \end{align*}
while all the other entries are equal to zero. This means that the input is applied to the first 5 masses on the left-hand side of the considered $n$-mass oscillator, hence the masses on the left-hand side od the oscillator have a larger influence on the input.

Moreover, we are interested in the 10 displacements and velocities in the middle of the $n$-mass oscillator. This means that the matrices $C_1,C_2 \in \mathbb{R}^{10\times n}$ from \eqref{eq:main} are defined by
 \begin{align*}
  C_1(1:10,46:55)&=C_2(1:10,46:55)=I,
 \end{align*}
with all other entries being equal to zero.

First, in Figure \ref{comp-Fval-diffp} we illustrate the behaviour of the optimal parameters and the magnitude of the
\emph{$p$--mixed $H_2$ norm} of the system defined by \eqref{eq:mixed_norm}. In more details, for given damping positions $i,j\in\{1,2,\ldots,n\}$ with $i<j$, we have optimized the viscosities and the optimal \emph{$p$--mixed $H_2$ norm} is plotted.  Figure \ref{comp-Fval-diffp} presents four different subplots that correspond to parameters $p=0,\frac{1}{3},\frac{2}{3},1$.  Here, for each damping configuration $(i,j)$ where $i<j$, we have optimized the viscosities and the optimal viscosities are denoted by $(v_1^0, v_2^0)$ while the optimal position for given $p$ is denoted by  $(i^0,j^0)$. In this figure we can also see the magnitudes of optimal damping parameters for considered parameters $p$.

\begin{figure}[h!]
\centering
\includegraphics[width=16cm]{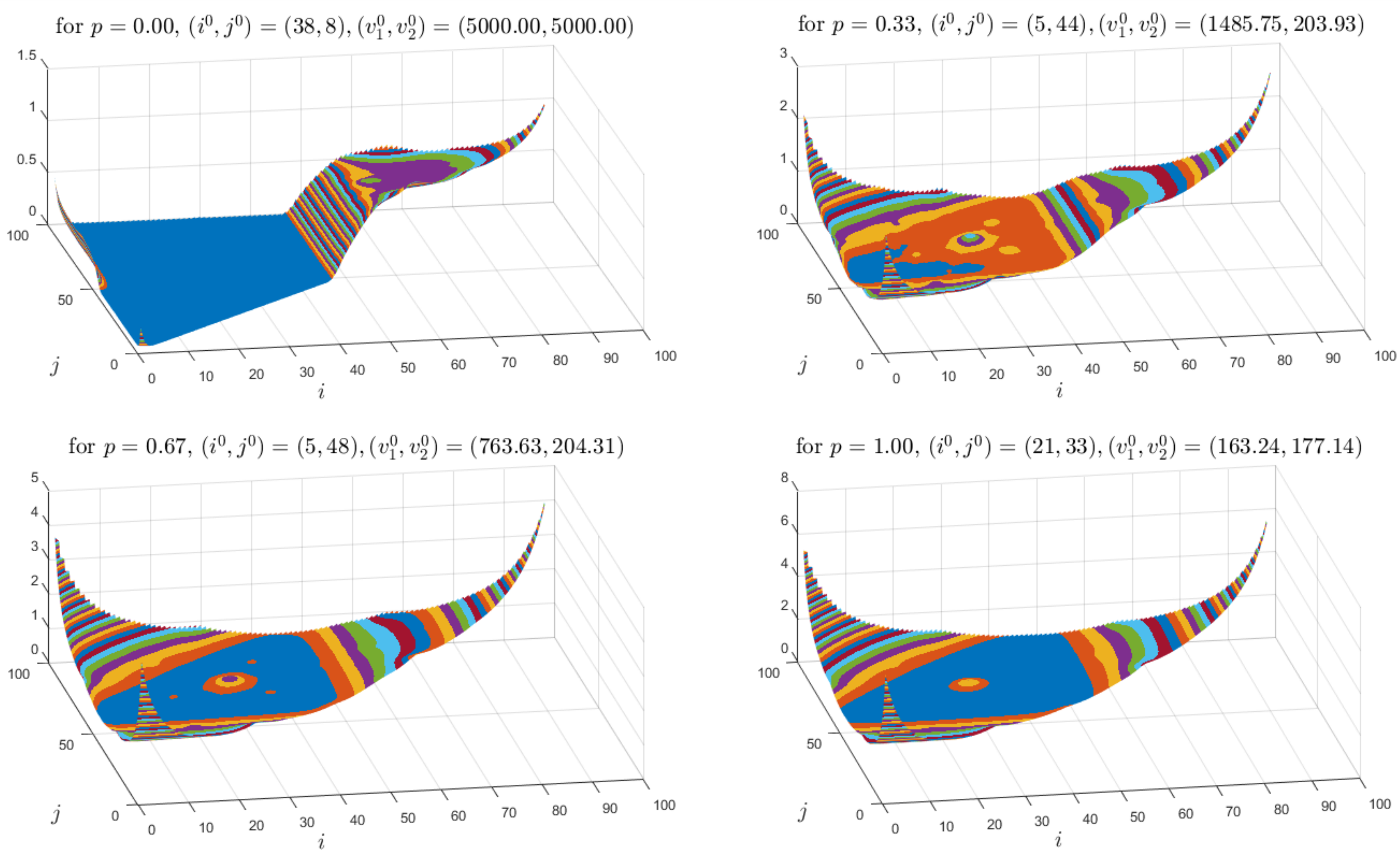}
 \caption{Comparison of the optimal \emph{$p$--mixed $H_2$ norm}  (given by \eqref{eq:mixed_norm})  for different parameters $p$}
 \label{comp-Fval-diffp}
\end{figure}

Figure \ref{comp-x1-diffp} shows similar results, but here on the $z$ axis instead of the \emph{$p$--mixed $H_2$ norm} we present the optimal first viscosity. This means that for each damping configuration $(i,j)$ where $i<j$, we have optimized viscosities $v_1$ and $v_2$ and the optimal first viscosity ($v_1^0$) is plotted. Similarly we can plot the optimal second viscosity, but surface plots illustrate a similar behaviour.

\begin{figure}[h!]
\centering
\includegraphics[width=16cm]{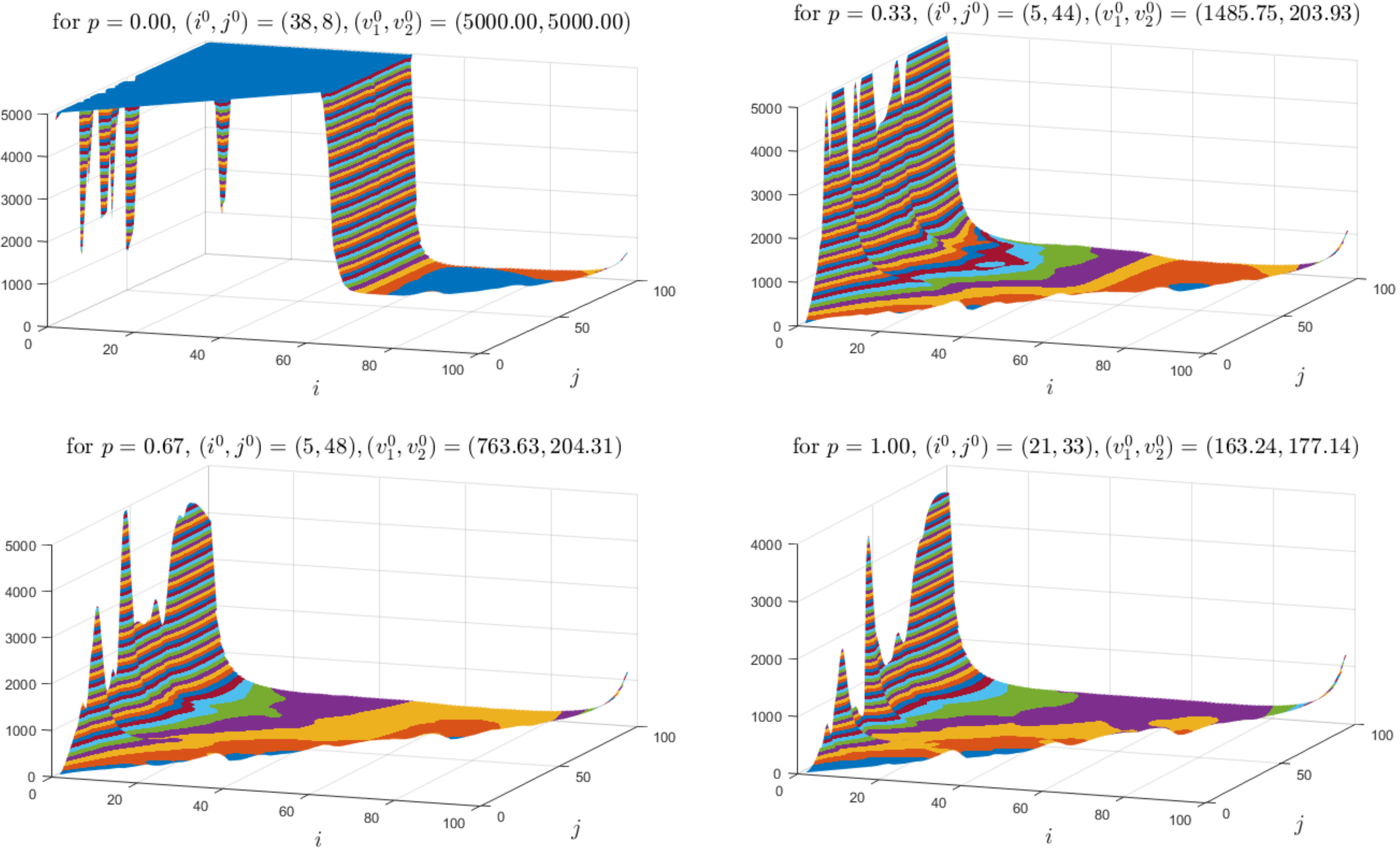}
 \caption{Comparison of the optimal first viscosity $v_1$ for different parameters $p$}
 \label{comp-x1-diffp}
\end{figure}

From  Figures \ref{comp-Fval-diffp} and \ref{comp-x1-diffp} we can conclude that the parameter $p$ has a very strong influence on optimal damping parameters. First, we can see that optimal damping positions vary significantly, while we change the parameter $p$ and the difference in optimal viscosities is even bigger. Moreover, for $p =0$ from Figure \ref{comp-x1-diffp} (which corresponds to the standard $H_2$ norm), we can see that there exists a whole area where the optimal first viscosity was equal to $5000$ (the right hand side of the feasible segment), which means that the optimal viscosity  should be as large as possible.

Let us now consider the second case where we consider the new performance measure given by \eqref{eq:mixed_norm_final}.

Now, we will consider the same example, but thew optimization criterion will be based on matrices given by \eqref{B2Cmatrices}. This means that we will consider the criterion given by \eqref{eq:mixed_norm_final}.

The same optimization process was performed as in the first case. First  of all, we should emphasize that we did not obtain a significant difference in optimal damping parameters. In particular, for all parameters $p=0,\frac{1}{3},\frac{2}{3},1$ the optimal configuration of damping positions was equal to $(27,53)$. On the other hand, we observe a slight change in optimal parameters, that is, optimal viscosities are (234.57,222.08), (229.05,217.41), (225.99,214.72) and (224.01,213.06) for $p=0,\frac{1}{3},\frac{2}{3},1$, respectively.

\begin{figure}[h!]
\centering
\includegraphics[width=16cm]{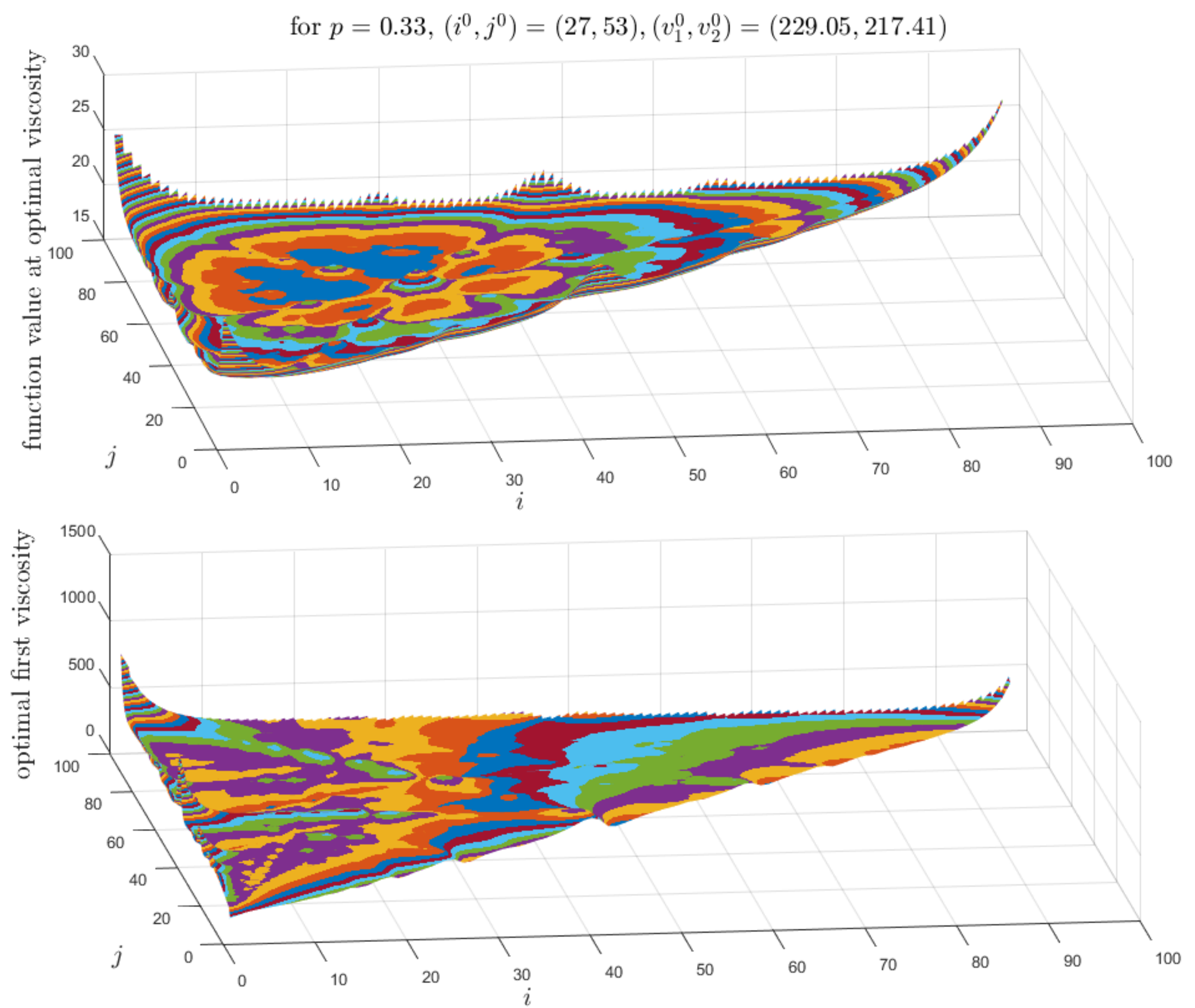}
 \caption{Comparison of the optimal \emph{$p$--mixed $H_2$ norm}  (given by \eqref{eq:mixed_norm_final}) and the optimal first viscosity $v_1$ for parameter $p=\frac{1}{3}$}
 \label{comp-x1Fval-diffp}
\end{figure}

In this case there isn't a significant difference in the surface plot for this penalty function. Thus, on Figure \ref{comp-x1Fval-diffp} we present surface subplots obtained by using the parameter $p=\frac{1}{3}$. On the first subplot we  present the surface plot for the new performance measure given by \eqref{eq:mixed_norm_final}, while on the second subplot we show magnitudes of the optimal first viscosity.


In this case our performance measure given by \eqref{eq:mixed_norm_final} is not significantly influenced by a change of the parameter $p$. Hence in this example the optimal damping parameters do not depend on the particular choice of the performance measure as long as $p>0$. This shows that for some vibrational systems the choice of the parameter $p$ is not important.
}
\end{example}

\begin{example}\label{ex2}
{\em In the second example we will consider a five-story shear frame structure from \cite{Xu:04} which is shown on Figure \ref{frame}. Similarly to the first example, in this example  the structure  can be modeled as a lumped mass system with equation (\ref{eq:main})
where  $$M=\diag(m_1,m_2,\ldots,m_5),\ K=\left[\begin{array}{ccccc}k_1+k_2&-k_2&0&0&0\\-k_2&k_2+k_3&-k_3&0&0\\0&-k_3&k_3+k_4&-k_4&0\\0&0&-k_4&k_4+k_5&-k_5\\0&0&0&-k_5&k_5\end{array}\right].$$

\begin{figure}[ht]
\centering
\begin{picture}(60,90)
\put(0,0){\scalebox{0.2}{\includegraphics{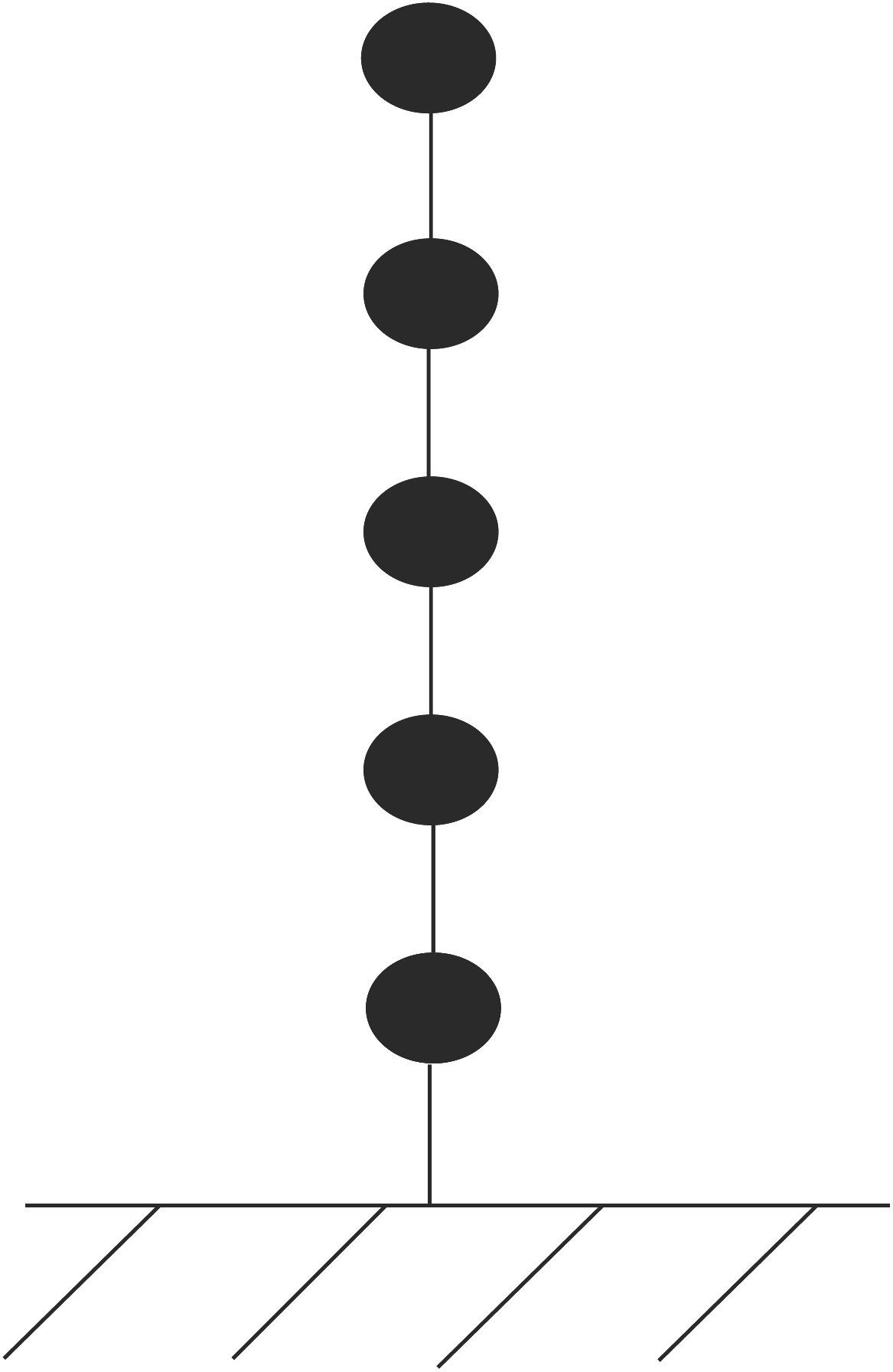}}}
\put(39,25){$m_1$}
\put(23,34){$k_2$}
\put(39,43){$m_2$}
\put(23,52){$k_3$}
\put(39,61){$m_3$}
\put(23,90){$k_5$}
\put(39,80){$m_4$}
\put(23,70){$k_4$}
\put(39,100){$m_5$}
\put(23,15){$k_1$}
\end{picture}
\vspace*{-0.5cm}
\caption{ A five-story frame structure }\label{frame}
\end{figure}

Configuration of masses and stiffnesses in this example corresponds to a structure from  \cite{Xu:04}  with mass and stiffness parameters given in Table 1.
\begin{table}[h!]\begin{center}
\begin{tabular}{|l|l|l|l|l|l|}
\hline
$i$th index&1&2&3&4&5\\\hline
$m_i$&4000&3000&2000&1000&800\\\hline
$k_i$&$3.375\times 10^6$&$3.75\times 10^6$&$3.375\times 10^6$&$3\times 10^6$&$2.25\times 10^6$\\\hline
\end{tabular}
\caption{ Mass and stiffness parameters of the structure} \end{center}
\end{table}

The damping matrix is modeled as
$ D =D_{\mathrm{int}} + D_{\mathrm{ext}},$  where the  internal damping $D_{int}$ is again a  small multiple of the critical damping as it was in the first example. In this example   we consider a damper that connects second and third floor of the five-story shear frame structure, therefore
external damping   is defined by $D_{ext}=v  (e_{2}-e_{3})(e_{2}-e_{3})^T$, where the parameter $v$ represents the viscosity parameter.

Similarly to the previous example, we are interested in the damping of all undamped eigenfrequencies.

Here the input matrix from  \eqref{eq:main} is determined by the matrix $B_2=\left(%
\begin{array}{ccccc}
   5000 & 0 & 0 & 0 &0
\end{array}%
\right)^T$  which means that the input is applied on the mass that is closest to the ground. This is natural since that this corresponds to disturbance that comes from ground. On the other hand, we will be interested in the stabilization of oscillations at the highest floor and therefore   we are interested in the   displacement  and velocity   of the highest mass, hence we take
 \begin{align*}
  C_1 &=C_2 =\left(%
\begin{array}{ccccc}
   0 & 0 & 0 & 0 &100
\end{array}%
\right).
 \end{align*}

In Figure \ref{opt gainsEx2} we show results for the new performance measure   where the penalty function corresponds to the \emph{$p$--mixed $H_2$ norm} of the system defined by \eqref{eq:mixed_norm}. The parameter $p$ varies from $0$ to $1$ and the first subplot shows the magnitude of the optimal viscosity parameter, while the second subplot shows the magnitude of the \emph{$p$--mixed $H_2$ norm} at the optimal viscosity. We can see from this figure the behaviour of the optimal value and the optimal function value. In particular, the optimal viscosity varies from  $1.09\cdot 10^5$  to $1.44\cdot 10^5$ which means that the parameter $p$ has significant influence on the magnitude of optimal viscosities.

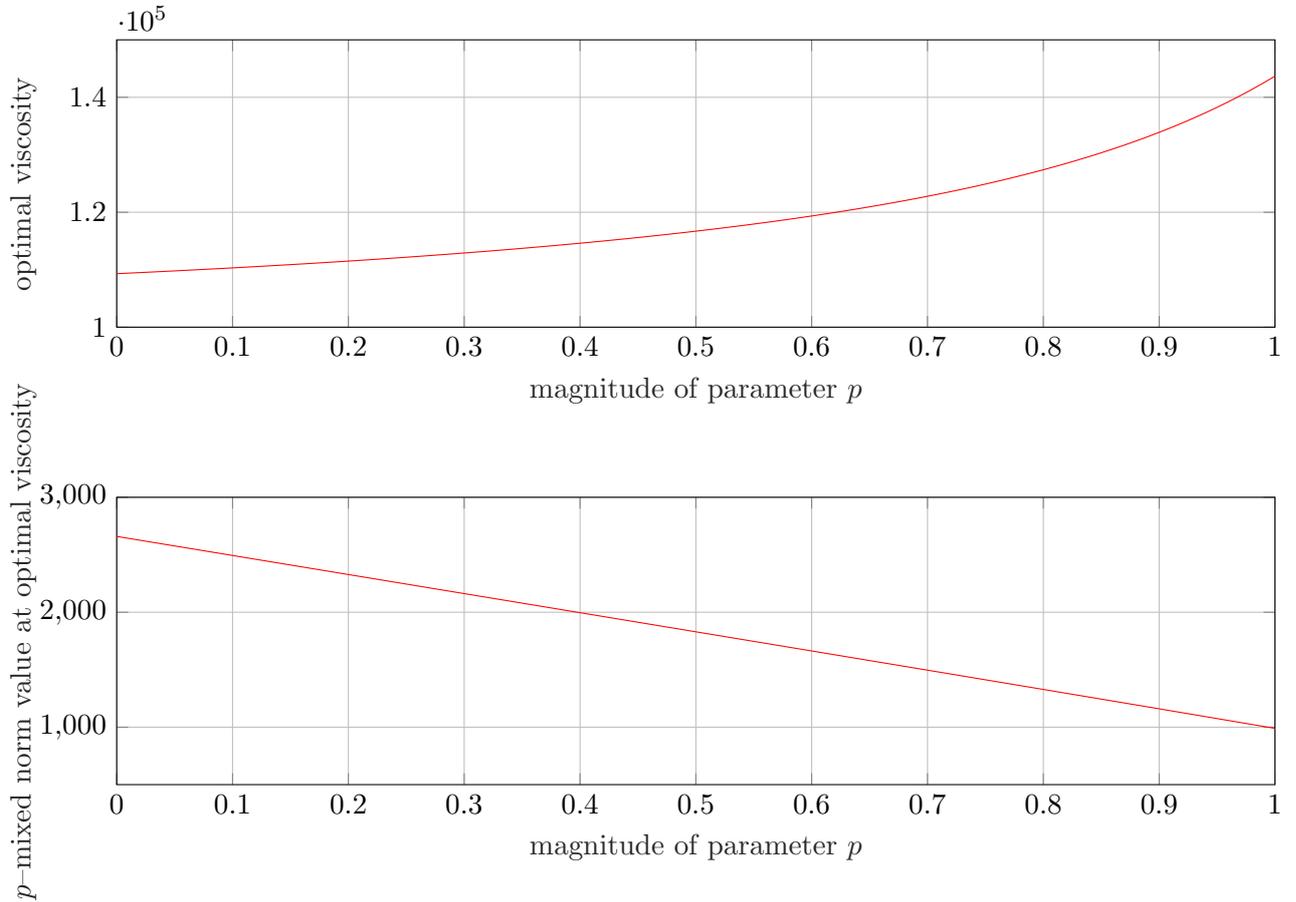
\begin{figure}[h!]
\centering
 \begin{tikzpicture}

\begin{axis}[%
width=6in,
height=1.5in,
at={(1.15in,3.588in)},
scale only axis,
xmin=0,
xmax=1,
xlabel style={font=\color{white!15!black}},
xlabel={magnitude of parameter $p$},
ymin=100000,
ymax=150000,
ylabel style={font=\color{white!15!black}},
ylabel={optimal viscosity},
axis background/.style={fill=white},
xmajorgrids,
ymajorgrids
]
\addplot [color=red, forget plot]
  table[row sep=crcr]{%
0	109308.106221657\\
0.01	109402.145638297\\
0.02	109497.578479054\\
0.03	109594.481891841\\
0.04	109692.851050877\\
0.05	109792.728423428\\
0.06	109894.15401294\\
0.07	109997.170548372\\
0.08	110101.798563116\\
0.09	110208.099728759\\
0.1	110316.098127123\\
0.11	110425.816601492\\
0.12	110537.3350087\\
0.13	110650.673434189\\
0.14	110765.866647628\\
0.15	110882.990200161\\
0.16	111002.079834108\\
0.17	111123.168813881\\
0.18	111246.328349108\\
0.19	111371.615480722\\
0.2	111499.053435808\\
0.21	111628.73054777\\
0.22	111760.683330203\\
0.23	111894.987318861\\
0.24	112031.709936275\\
0.25	112170.884273715\\
0.26	112312.6125002\\
0.27	112456.93925442\\
0.28	112603.963696741\\
0.29	112753.727148988\\
0.3	112906.339022636\\
0.31	113061.839862672\\
0.32	113220.353628119\\
0.33	113381.94736057\\
0.34	113546.715761311\\
0.35	113714.746088491\\
0.36	113886.142034595\\
0.37	114061.003069035\\
0.38	114239.419761214\\
0.39	114421.522959734\\
0.4	114607.405505777\\
0.41	114797.212912724\\
0.42	114991.045066954\\
0.43	115189.032474284\\
0.44	115391.319030292\\
0.45	115598.028927127\\
0.46	115809.301971041\\
0.47	116025.32437313\\
0.48	116246.210884609\\
0.49	116472.163615848\\
0.5	116703.337430556\\
0.51	116939.905797804\\
0.52	117182.061004189\\
0.53	117430.014343157\\
0.54	117683.956146624\\
0.55	117944.103148304\\
0.56	118210.698433734\\
0.57	118483.95393749\\
0.58	118764.142608093\\
0.59	119051.508343881\\
0.6	119346.328860569\\
0.61	119648.915978482\\
0.62	119959.512550896\\
0.63	120278.514351062\\
0.64	120606.210699289\\
0.65	120942.945280289\\
0.66	121289.126272581\\
0.67	121645.123617467\\
0.68	122011.347387022\\
0.69	122388.25935426\\
0.7	122776.28014549\\
0.71	123175.934304159\\
0.72	123587.747090335\\
0.73	124012.195691176\\
0.74	124449.948056481\\
0.75	124901.536299536\\
0.76	125367.680464915\\
0.77	125849.033043466\\
0.78	126346.333318826\\
0.79	126860.392924578\\
0.8	127391.989197424\\
0.81	127942.045047743\\
0.82	128511.499362341\\
0.83	129101.374197363\\
0.84	129712.70610456\\
0.85	130346.695436746\\
0.86	131004.528246337\\
0.87	131687.577650892\\
0.88	132397.222513392\\
0.89	133134.999252789\\
0.9	133902.548199762\\
0.91	134701.610971867\\
0.92	135534.103189571\\
0.93	136402.068536385\\
0.94	137307.703244199\\
0.95	138253.422095376\\
0.96	139241.766690073\\
0.97	140275.586657073\\
0.98	141357.909026034\\
0.99	142492.037863274\\
1	143681.587752869\\
};
\end{axis}

\begin{axis}[%
width=6in,
height=1.5in,
at={(1.15in,1.2in)},
scale only axis,
xmin=0,
xmax=1,
xlabel style={font=\color{white!15!black}},
xlabel={magnitude of parameter  $p$},
ymin=500,
ymax=3000,
ylabel style={font=\color{white!15!black}},
ylabel={$p$--mixed norm value at optimal viscosity},
axis background/.style={fill=white},
xmajorgrids,
ymajorgrids
]
\addplot [color=red, forget plot]
  table[row sep=crcr]{%
0	2659.61144643095\\
0.01	2643.04315157305\\
0.02	2626.47401453936\\
0.03	2609.90401732258\\
0.04	2593.33314140105\\
0.05	2576.76136772034\\
0.06	2560.18867667411\\
0.07	2543.61504808403\\
0.08	2527.04046117885\\
0.09	2510.46489457273\\
0.1	2493.88832624234\\
0.11	2477.31073350315\\
0.12	2460.73209298449\\
0.13	2444.15238060366\\
0.14	2427.57157153869\\
0.15	2410.98964020005\\
0.16	2394.40656020079\\
0.17	2377.82230432556\\
0.18	2361.23684449797\\
0.19	2344.65015174659\\
0.2	2328.0621961691\\
0.21	2311.47294689504\\
0.22	2294.88237204643\\
0.23	2278.29043869669\\
0.24	2261.6971128275\\
0.25	2245.10235928343\\
0.26	2228.5061417245\\
0.27	2211.90842257614\\
0.28	2195.30916297679\\
0.29	2178.70832272273\\
0.3	2162.10586021012\\
0.31	2145.50173237398\\
0.32	2128.89589462409\\
0.33	2112.28830077735\\
0.34	2095.67890298669\\
0.35	2079.0676516661\\
0.36	2062.45449541159\\
0.37	2045.83938091794\\
0.38	2029.22225289073\\
0.39	2012.60305395359\\
0.4	1995.98172455027\\
0.41	1979.35820284109\\
0.42	1962.73242459364\\
0.43	1946.10432306703\\
0.44	1929.47382888958\\
0.45	1912.84086992929\\
0.46	1896.20537115667\\
0.47	1879.56725449939\\
0.48	1862.92643868826\\
0.49	1846.28283909377\\
0.5	1829.63636755274\\
0.51	1812.9869321841\\
0.52	1796.33443719339\\
0.53	1779.67878266468\\
0.54	1763.01986433953\\
0.55	1746.35757338154\\
0.56	1729.69179612574\\
0.57	1713.02241381158\\
0.58	1696.34930229815\\
0.59	1679.67233176052\\
0.6	1662.99136636547\\
0.61	1646.30626392519\\
0.62	1629.61687552706\\
0.63	1612.92304513779\\
0.64	1596.22460917954\\
0.65	1579.52139607613\\
0.66	1562.81322576653\\
0.67	1546.09990918307\\
0.68	1529.38124769145\\
0.69	1512.65703248914\\
0.7	1495.92704395878\\
0.71	1479.19105097249\\
0.72	1462.44881014305\\
0.73	1445.70006501685\\
0.74	1428.94454520387\\
0.75	1412.18196543858\\
0.76	1395.41202456569\\
0.77	1378.63440444366\\
0.78	1361.84876875847\\
0.79	1345.05476173892\\
0.8	1328.2520067642\\
0.81	1311.4401048534\\
0.82	1294.61863302525\\
0.83	1277.78714251534\\
0.84	1260.94515683675\\
0.85	1244.09216966796\\
0.86	1227.22764255072\\
0.87	1210.35100237811\\
0.88	1193.46163865078\\
0.89	1176.55890047714\\
0.9	1159.64209328991\\
0.91	1142.7104752485\\
0.92	1125.76325329285\\
0.93	1108.79957881022\\
0.94	1091.8185428716\\
0.95	1074.81917098913\\
0.96	1057.80041733962\\
0.97	1040.76115839244\\
0.98	1023.70018587194\\
0.99	1006.61619897559\\
1	989.507795758491\\
};
\end{axis}
\end{tikzpicture}
\caption{Comparison of optimal gains and optimal function values for the second example} \label{opt gainsEx2}
\end{figure}
}
\end{example}

\medskip

\noindent\textbf{Acknowledgments}\\
Supported in part by the National Science Foundation under the project ``Control of Dynamical Systems'', Grant No.\ IP-2016-06-2468 and under the project ``Optimization of parameter dependent mechanical systems'' (IP-2014-09-9540), Grant No.\ 9540.

\bibliographystyle{plain}
\bibliography{bibliogr_opt_ab}

\end{document}